\pdfoutput=1
\documentclass[12pt,final]{article}

\usepackage[utf8]{inputenx}
\usepackage[T1]{fontenc}

\usepackage{lmodern}           
\usepackage[scaled]{beramono}  
\usepackage[final]{microtype}  
\usepackage{siunitx}           
\usepackage{listings}          
\usepackage{url}               
\usepackage[shortlabels]{enumitem}
\usepackage{amssymb}    
\usepackage{amsthm}     
\usepackage{thmtools}   
\usepackage{mathtools}  
\usepackage{mathrsfs}   
\usepackage{dsfont}     
\usepackage{cancel}     
\usepackage{stmaryrd}   

\usepackage{xspace}         
\usepackage[top=3.5cm,bottom=4cm]{geometry}

\usepackage[backend = biber, style = alphabetic]{biblatex}

\usepackage{varioref}
\usepackage{hyperref}
\usepackage[nameinlink, capitalize, noabbrev]{cleveref}

\usepackage{tikz-cd}
\usepackage{enumitem}
\usepackage{todonotes}
\usepackage{showlabels}

\declaretheorem[style = plain, numberwithin = section]{theorem}
\declaretheorem[style = plain,      sibling = theorem]{corollary}
\declaretheorem[style = plain,      sibling = theorem]{lemma}
\declaretheorem[style = plain,      sibling = theorem]{proposition}

\declaretheorem[style = definition, sibling = theorem]{definition}

\crefname{observation}{Observation}{Observations}
\Crefname{observation}{Observation}{Observations}
\crefname{conjecture}{Conjecture}{Conjectures}
\Crefname{conjecture}{Conjecture}{Conjectures}
\crefname{notation}{Notation}{Notations}
\Crefname{notation}{Notation}{Notations}
\Crefname{claim}{Claim}{Claims}
\Crefname{question}{Question}{Questions}

\newcommand{\CY}{Calabi--\kern-0.37exYau\xspace}

\newcommand{\eg}{\leavevmode\unskip, e.g.,\xspace}

\newcommand{\Z}{\mathbb{Z}}    
\newcommand{\Q}{\mathbb{Q}}    
\newcommand{\C}{\mathbb{C}}    
\newcommand{\F}{\mathbb{F}}    
\newcommand{\A}{\mathbb{A}}    
\renewcommand{\P}{\mathbb{P}}  

\DeclarePairedDelimiter{\set}{\lbrace}{\rbrace}        
\DeclarePairedDelimiter{\abs}{\lvert}{\rvert}          

\newcommand{\from}{\colon}

\newcommand{\sO}{\mathscr{O}}

\newcommand\restr[2]{{
  \left.\kern-\nulldelimiterspace 
  #1 
  \vphantom{\big|} 
  \right|_{#2} 
}}

\DeclareMathOperator{\CH}{CH}

\DeclareMathOperator{\Spec}{Spec}
\DeclareMathOperator{\chr}{char}
\DeclareMathOperator{\trdeg}{trdeg}

\usepackage{titlesec}
\titleformat*{\section}{\large\bfseries}
\titleformat*{\subsection}{\bfseries}
\titleformat*{\subsubsection}{\small\bfseries}

\title{\vspace{-3.5cm}\Large{\bfseries{A Stably Irrational (2,3)-Complete Intersection Fourfold over $\Q$}}} 
\author{Bjørn Skauli}
\date{}
\begin{document}
\maketitle

\begin{abstract}
We apply the specialization technique based on the decomposition of the diagonal from \cite{Tot16} and \cite{SchHypersurface} to find an explicit example over $\Q$ of a quadric and cubic hypersurface in $\P^6$ such that their intersection is a smooth stably irrational fourfold. Using the same degeneration, Nicaise and Ottem (\cite{NO20}) have already proven that the the very general complete intersection of this type is stably irrational using the motivic volume.
\end{abstract}

\section{Introduction}
Determining which varieties are birational to projective space is a central problem in birational geometry. In dimension 1 and 2 the problem was solved by Lüroth and Castelnuovo respectively, but in higher dimensions the rationality problem has proven to be harder. In the study of the rationality problem, several weaker notions of rationality have been introduced. Perhaps the most important is stable rationality, where a variety $X$ is stably rational if $X \times \P^n$ is birational to $\P^m$ for some $n,m \geq 0$. In \cite{Voi15} Voisin introduced a degeneration technique based on the decomposition of the diagonal to prove stable irrationality of the very general quartic double solid. Since then such degeneration techniques have been a powerful tool used to prove stable irrationality of many classes of varieties. In work by Totaro \cite{Tot16} and Schreieder \cite{SchHypersurface}, the technique has been used to even find explicit examples over $\Q$ of stably irrational hypersurfaces.

 A different degeneration technique was introduced by Nicaise and Schinder in \cite{NS19}, based on the motivic volume. In \cite{NO20}, Nicaise and Ottem use this technique to prove that the very general complete intersection of a cubic and a quadric in $\P^6$ is stably irrational. However, the stable birational volume technique is less suited to finding explicit examples. The goal of this paper is to use techniques introduced by Totaro and Schreieder in \cite{Tot16} and \cite{SchHypersurface} to find an example of a $(2,3)$-complete intersection in $\P^6$ with integer coefficients that is not stably rational.
 
Specifically, we will prove the following result:
\begin{theorem}
	\label{thm:intro}
	Let $K=\Q$ or $K=\F_p(t)$. In the first case let $p,q \geq 5$ be distinct primes and set $u=p,v=q$, and in the second case let $u=t,v=(t-1)$. Let $X \subset \P^6_K$ be the complete intersection defined by the following two equations:
	\begin{equation}
		\label{eq:quadricIntro}
		\ u(\sum_{i=0}^6 x_i^2) + v(x_3x_6-x_4x_5)=0
	\end{equation}
	
	\begin{equation}
			 u(\sum_{i=0}^6 x_i^3) + v(x_0^2x_5 + x_1^2x_4 + x_2^2x_6 + x_3(x_5^2+x_4^2+x_3^2 -2x_3(x_6 + x_5 + x_4)))=0.
	\end{equation}
Then $X$ is a smooth complete intersection that is not geometrically stably rational.
\end{theorem}
The unifying property of the two choices for $K$ is that varieties over $K$ can be specialized to varieties over $\F_p$.

In \cref{sec:Prelim} we collect the important definitions and results we will use. Then in \cref{sec:Example} we will prove that the complete intersection in \cref{thm:intro} is geometrically stably irrational. To do this we specialize the complete intersection to the union of two components, such that one component is birational to the stably irrational quadric bundle found in \cite{HPT19} and the intersection of the components is rational. This is the same specialization as the one used in \cite{NO20}

\section{Rationality and specialization}
\label{sec:Prelim}
\subsection{Unramified cohomology}
Unramified cohomology groups are subgroups of the \'etale cohomology groups. If $X$ is a scheme and $F$ a sheaf on $X$ in the \'etale topology of $X$, we denote the $i$-th \'etale cohomology group by $H^i(X,F)$. If $R$ is a ring we will use $H^i(R,F)$ as a shorthand for $H^i(\Spec R,F)$.

We refer to \cite{SchSurvey} for an introduction to unramified cohomology. Following \cite{SchSurvey} and \cite{Mer08} we define unramified cohomology using only geometric valuations:

\begin{definition}{{\cite[Definition 4.3]{SchSurvey}}}
  Let $K/k$ be a finitely generated field extension and let $m$ be a positive integer that is invertible in $k$. We define the unramified cohomology of $K$ over $k$ with coefficients on $\mu_m^{\otimes j}$ as the subgroup
  \[H_{nr}^i(K/k,\mu_m^{\otimes j}) \subset H^i(K,\mu_m^{\otimes j}) \]
consisting of all elements $\alpha \in H^i(K,\mu_m^{\otimes j})$ such that for any geometric valuation $\nu$ on $K$ over $k$ we have $\partial_\nu (\alpha) = 0$.
\end{definition}

\begin{definition}{{\cite[Definition 4.1]{SchSurvey}}}
	Let $K/k$ be a finitely generated field extension. A \emph{geometric valuation} $\nu$ on $K$ over $k$ is a discrete valuation on $K$ over $k$ such that the transcendence degree of $\kappa_\nu$, the field of fractions of the corresponding DVR, over $k$ is given by
	\[\trdeg_k(\kappa_\nu) = \trdeg(K)-1\]
\end{definition}
The main reason for using geometric valuations is that it gives better functorial properties. Importantly, if $K'/K$ is a field extension over a field $k$, the corresponding pullback maps on \'etale cohomology restrict to a pullback map on unramified cohomology. If $K'/K$ is a finite extension, the pushforward in \'etale cohomology gives a pushforward on unramified cohomology groups.

We can define restrictions of unramified cohomology classes to scheme points. To do this, we need the so-called injectivity and codimension 1 purity properties for \'etale cohomology, which are consequences of Bloch-Ogus' proof of the Gersten conjecture (\cite{BO74}). See \cite[Theorem 3.6]{SchSurvey} or \cite[Theorems 3.81. and 3.8.2]{CT95}

\begin{theorem}
  \label{thm:EtaleLocGlob}
  Let $X$ be a variety over a field $k$ and let $m$ be a positive integer that is invertible in $k$. Let $x$ be a point in the smooth locus of $X$. Then the following holds:
  \begin{enumerate}[i)]
  \item The natural morphism
    \begin{equation}
      \label{eq:3.7}
          H^i(\sO_{X,x},\mu_m^{\otimes j}) \to H^i(k(X),\mu_m^{\otimes j})
    \end{equation}
    is injective
  \item A class $\alpha \in H^i(k(X),\mu_m^{\otimes j} )$ lies in the image of \eqref{eq:3.7} if and only if $\alpha$ has trivial residue along each prime divisor on $X$ that passes through $x$.
  \end{enumerate}
\end{theorem}

We can now define the restriction of an unramified cohomology class. The definition here is stated slightly more generally than \cite[Proposition 4.8]{SchSurvey} but with the same proof.

\begin{proposition}
	\label{prop:UnramifiedRestriction}
  Let $X$ be a variety over a field $k$ and let $m$ be a positive integer that is invertible in $k$. Let $\alpha \in H_{nr}^i(k(X)/k, \mu_m^{\otimes j})$.
  \begin{enumerate}[i)]
  \item
   Let $x$ be a scheme point in the smooth locus of $X$. Then there is a well-defined restriction
    \[\restr{\alpha}{x} \in H^i(\kappa(x),\mu_m^{\otimes j}). \]
  \item
    If $X$ is also smooth and proper over $k$, then $\restr{\alpha}{x} \in H^i(\kappa(x),\mu_m^{\otimes j})$ is unramified over $k$.
  \end{enumerate}
\end{proposition}

\subsection{Decomposition of the Diagonal}
The decomposition of the diagonal technique was introduced in \cite{BS83}, and its use in answering questions of stable rationality developed by among others \cite{Voi15}, \cite{CTP16} \cite{Tot16}, \cite{SchHypersurface}, \cite{SchTorsion}.

\begin{definition}
We say a scheme of pure dimension $n$ over a field $k$ admits a \emph{decomposition of the diagonal} if we have an equality:
\[\Delta_X = X \times z  + Z_X \in CH_n(X \times X) \]
where $Z_X$ is a cycle supported on $D \times X$ for some divisor $D \subset X$ and $z \in Z_0(X)$ is a zero-cycle on $X$.
\end{definition}
It will often be convenient to look at decompositions of the diagonal in the following way:
\begin{lemma}{{\cite[Lemma 7.3]{SchSurvey}}}
	A variety $X$ over a field $k$ admits a decomposition of the diagonal if and only if there is a $0$-cycle $z \in Z_0(X)$ such that:
	\[[\delta_X] = [z \times {k(X)}] \in CH_0(X \times {k(X))}\]
	where we write $z \times k(X)$ for the base change of $z$ to $X \times_k \Spec k(X)$.
\end{lemma}
The equivalence follows from the natural isomorphism
\[\varinjlim_{\emptyset \neq U \subset X} CH_n(U \times_k X) \simeq CH_0(X \times {k(X)})\]

The following lemma relates decompositions of the diagonal to stable rationality:
\begin{lemma}{{\textrm{(See \eg \cite[Lemma 2.4]{SchHypersurface})}}}
	A variety $X$ over a field $k$ that is stably rational admits a decomposition of the diagonal.
\end{lemma}

\subsection{The Merkurjev Pairing}
We will use the Merkurjev pairing introduced in \cite[Section 2.4]{Mer08} to detect whether a smooth variety has a decomposition of the diagonal.

\begin{proposition}
	\label{prop:MerkurjevPairing}
	Let $X$ be a smooth proper variety over a field $K$ (not necessarily algebraically closed) and let $m$ be an integer invertible in $K$. Then there is a bilinear pairing:
	\[CH_0(X) \times H_{nr}^i(k(X)/K,\mu_m^{\otimes j}) \to H^i(K,\mu_m^{\otimes j})\]
	which we will write as $(z,\alpha) \to \langle z,\alpha \rangle$. For a closed point $z$ the pairing is given by:
	\[\langle z,\alpha \rangle = i_*(\restr{\alpha}{Z}) \in H^i(K,\mu_m^{\otimes j})\]
	for $i \from \Spec k(z) \to X$.
\end{proposition}

\subsection{Alterations in characteristic $p$}
The Merkurjev pairing works on smooth varieties, but since resolution of singularities is still unknown in positive characteristic we will need to use so-called alterations:
\begin{definition}
	Let $Y$ be a variety over an algebraically closed field $k$. An \emph{alteration} of $Y$ is a proper generically finite surjective morphism $Y' \to Y$, where $Y'$ is a non-singular variety over $k$.
\end{definition}
By de Jong \cite{deJ96}, alterations exist in any characteristic and by work of Gabber the degree of the alteration can be chosen to be coprime with any prime not dividing the characteristic of the field. In fact Temkin proves that one can choose the degree to be a power of the characteristic \cite[Theorem 1.2.5]{Tem17}(or degree 1 if $\chr(k)=0$).

\subsection{Specialization of varieties over $\Q$ or $\F_p(t)$}
Specialization of a decomposition of the diagonal can be used to find examples of varieties over $\Q$ that are not geometrically rational. The same argument will work over field of characteristic $p > 2$ with at least one transcendent element $t$ over $\F_p$. The following is the precise result we will use, where the proof is included to explain the technique. Both the statement and proof are adapted from \cite[Corollary 8.3]{SchSurvey}.

\begin{proposition}{{(cf. \cite[Corollary 8.3]{SchSurvey})}}
	\label{prop:ZSpecialization}
	Let $R = \Z_{(p)}$ or $R = \F_p[t]_{(t)}$, with field of fractions $K=\Q$ or $K=\F_p(t)$ respectively and residue field $\mathbb{F}_p$. Let $\mathscr{X} \to \Spec R$ be a scheme with generic fibre $X_{K}$ of dimension $n$ and geometric special fibre $Y_{\overline{\mathbb{F}_p}}$. Assume that $X_{\overline{K}}$ admits a decomposition of the diagonal \eg when $X_K$ is geometrically stably rational. Then the geometric special fibre $Y_{\overline{\mathbb{F}_p}}$ also admits a decomposition of the diagonal.
\end{proposition}
\begin{proof}
  Let $\mathscr{X}' \to R'$ be the base change of $\mathscr{X}$ to the completion $R'$ of $R$, and let $K'$ be the field of fractions of $R'$. Since $X_{\overline{K}}$ admits a decomposition of the diagonal, so does $X_{\overline{K'}}$, the geometric generic fibre of $\mathscr{X'}$. We get a relation in $CH_n(X_{\overline{K'}} \otimes_{\overline{K'}} X_{\overline{K'}})$
  \[ [\Delta_{X_{\overline{K'}}}] = [X_{\overline{K'}} \times z] + [Z_{\overline{K'}}]\]
  where $z$ is a zero-cycle on $X_{\overline{K'}}$, and $Z_{\overline{K'}}$ is supported on $D \times X_{\overline{K'}}$ for $D$ a divisor in $X_{\overline{K'}}$. In fact, there is a finite extension $L/K'$ such that the above relation holds over $L$. Since $\Z_p$ is complete, the integral closure of $R'$ in $L$ is a DVR, which we will denote by $R''$ with residue field $k$. The map $\Spec R'' \to \Spec \Z_p$ is finite, so after a finite base change $\mathscr{X}'' \to \mathscr{X}'$, we may assume that we have the relation:
  \[ [\Delta_{X_L}] = [X_{L} \times z] + [Z_{L}] \]
  in $\CH_0(X_L \times_L X_L)$.
  Consider $\mathscr{X}'' \times_R'' \mathscr{X}'' \to R''$. Fulton \cite[Chapter 20.3]{Ful98} defines a specialization map $\sigma \from \CH_0(X_L) \to \CH_0(Y)$ which on cycles acts by taking the closure of the pushforward, then pulling back to the special fibre.

  Applying this map to the relation above gives
  \[ \sigma([\Delta_{X_L}]) = \sigma([X_{L} \times z]) + \sigma([Z_{L}])\]
  which is equal to the relation
  \[ [\Delta_{Y_{k}}] = [Y_{k} \times z] + [Z_{k} \times Y_{k}]\]
  where $z$ is a zero-cycle on $Y_k$ and $Z_{k}$ is supported on a divisor in $Y_k$.

So the special fibre $Y_k$ has a decomposition of the diagonal. After a base change we get that also $Y_{\overline{\F_p}}$ has a decomposition of the diagonal.
\end{proof}

\section{A particular (2,3)-complete intersection}
\label{sec:Example}
We will apply this specialization technique to find a quadric and a cubic fivefold, defined over $\Q$. Using the specialization used in \cite{NO20}, the intersection the two hypersurfaces specializes to a variety birational to the quadric constructed in \cite{HPT19}, which has a non-trivial unramified cohomology class. From this it will follow that the original complete intersection is stably irrational.
\subsection{Constructing an example}
Let $R = \Z$ or $R=\F_p[t]$, with field of fractions $K$. If $R=\Q$ we pick any two distinct primes $p,q \geq 5$ and set $u=p,v=q$, otherwise we set $u=t,v=(t-1)$. We will consider the complete intersection $\mathscr{X} \coloneqq \mathscr{Q} \cap \mathscr{C} \subset \P_{\Z}^6$, where $\mathscr{Q}$ and $\mathscr{C}$ are the following hypersurfaces:
\begin{equation}
  \label{eq:quadric}
   \mathscr{Q} =V\left(  u(\sum_{i=0}^6 x_i^2) + v(x_3x_6-x_4x_5) \right)  \subset \P_{R}^6
\end{equation}

\begin{equation}
  \label{eq:cubic}
  \begin{split}
  \mathscr{C} = &V\left( u(\sum_{i=0}^6 x_i^3) + v(x_0^2x_5 + x_1^2x_4 + x_2^2x_6 + x_3(x_5^2+x_4^2+x_3^2 -2x_3(x_6 + x_5 + x_4))) \right) \\ &\subset \P_{R}^6
  \end{split}
\end{equation}
\begin{proposition}
  Let $\mathscr{X}$ be as above and let $X$ be the generic fibre of $\mathscr{X} \to \Spec R$, then $X$ is a smooth complete intersection in $\P_{K}^6$.
\end{proposition}
\begin{proof}
  Consider the scheme $\mathscr{X} \to \Spec R$.  If $R=\Q$, the fibre over $(q)$ is the intersection of the Fermat quartic and the Fermat cubic in $\P_{\F_q }^6$, which is smooth. Thus $X$ is smooth by generic smoothness. If $R=\F_p[t]$ we look at the fibre over the ideal $(t-1)$ and apply the same argument.
\end{proof}

Let $\mathfrak{p}$ be the ideal $(p)$ or $(t)$ depending on if $R$ is $\Z$ or $\F_p[t]$ respectively. Let $\mathscr{X} \to \Spec R_{\mathfrak{p}}$ be defined by the two equations \eqref{eq:quadric} and \eqref{eq:cubic}. The fibre $X_p$ above the closed point $\Spec \F_p$ in the DVR $R_{\mathfrak{p}}$ is the complete intersection in $\P^6_{\F_p}$ of the two hypersurfaces:
\begin{equation}
  \label{eq:Qquadric}
  Q_p = V(x_3x_6-x_4x_5)
\end{equation}
\begin{equation}
  \label{eq:Qqubic}
  C_p = V \left(x_0^2x_5 + x_1^2x_4 + x_2^2x_6 + x_3(x_5^2+x_4^2+x_3^2 -2(x_3x_6 + x_3x_5 + x_3x_4)) \right)
\end{equation}
We will prove that $X_p$ does not have a decomposition of the diagonal. Then, from \cref{prop:ZSpecialization} it will follow that $X$ is not geometrically stably rational over $\Q$.

The hypersurface $Q_p$ is the cone over $\P_{\F_p}^1 \times \P_{\F_p}^1$ embedded in the $\P_{\F_p}^3 \subset \P_{\F_p}^6$ with coordinates $x_3,x_4,x_5,x_6$. It is singular along the plane $V(x_3,x_4,x_5,x_6)$, which is the vertex of the cone.

The complete intersection $X_p = Q_p \cap C_p$ is singular along the plane $V(x_3,x_4,x_5,x_6)$. Additionally, it is singular along four curves: The plane conics defined by
\begin{gather*}
  V(x_1,x_5,x_6,x_3-x_4,x_0^2 + x_2^2 + x_3^2) \\
  V(x_0,x_4,x_6,,x_3 - x_5,x_1^2 + x_2^2 + x_3^2)
\end{gather*}
and the plane cubics defined by:
\begin{gather*}
  V(x_1,x_2,x_3,x_5,x_4^3+x_0^2x_6)\\
  V(x_0,x_2,x_3,x_4,x_5^3+x_1^2x_6)
\end{gather*}

\subsection{Proving stable irrationality}
The special fibre $X_p$ is very singular, which makes it more difficult to prove stable irrationality. The first step in alleviating this is:

\begin{lemma}
	The map $\pi \from P \coloneqq \P_{\P_{\F_p}^1 \times \P_{\F_p}^1}(\sO^{\oplus 3} \oplus \sO(1,1)) \to Q_p \subset \P_{\F_p}^6$ defined by the base-point-free linear system $\abs{\sO_P(1)}$ is the blow-up of $Q_p$ in the vertex plane $V(x_3,x_4,x_5,x_6)$.
\end{lemma}
\begin{proof}
	Let $y_0,y_1,z_0,z_1$ be coordinates on $\P_{\F_p}^1 \times \P_{\F_p}^1$, and $U,V,W,T$ be coordinates in the fibres of $P$. A basis for $H^0(\sO_P(1))$ is then: $\set{U,V,W,y_0z_0T,y_0z_1T,y_1z_0T,y_1z_1T}$. The corresponding map $P \to \P_{\F_p}^6$ has image $Q_p$. The projective bundle $P$ is smooth and $\pi^{-1}(V(x_3,x_4,x_5,x_6))$ is the subbundle $\P_{\P_{\F_p}^1 \times \P_{\F_p}^1}(\sO^{\oplus 3})$, which is a divisor. Hence $P$ is the blow-up of the vertex plane.
\end{proof}

Let $F$ be the following polynomial in $\abs{\sO_P(2) \otimes p^*(\sO_{\P_{\F_p}^1 \times \P_{\F_p}^1}(1,1))}$:
\begin{equation}
\begin{aligned}
  \label{eq:X1}
  F(y_0,y_1,z_0,z_1,U,V,W,T)&=y_0z_1U^2 + y_1z_0 V^2 + y_1z_1W^2 \\
   &+ y_0z_0(y_1^2z_0^2 + y_0^2z_1^2 + y_0^2z_0^2-2(y_1z_1 + y_1z_0 + y_0z_1)T^2
\end{aligned}
\end{equation}
Then $\pi^{-1}(X_p)$ is defined by $TF(y_0,y_1,z_0,z_1,U,V,W,T)=0$ which we recognize as having two components. We will denote the components by $X_1$ and $X_2$. Let $X_1$ be defined by $F(y_0,y_1,z_0,z_1,U,V,W,T)=0$, $X_2$ be defined by $T=0$, and $Z$ be the intersection of the two components.

The component $X_2$ is a projective bundle over $\P_{\F_p}^1 \times \P_{\F_p}^1$, and is therefore a smooth rational variety. Precisely, $X_2 \simeq \P_{\P_{\F_p}^1 \times \P_{\F_p}^1}(\sO^{\oplus 3})$.

The intersection $Z$ of the two components, defined by 
\[T=F(y_0,y_1,z_0,z_1,U,V,W,T)=0,\]
 is a conic bundle over $\P_{\F_p}^1 \times \P_{\F_p}^1$. The bundle is defined by the equation:
\begin{equation}
	\label{eq:Z1}
	y_0x_1U^2 + x_0y_1 V^2 + x_1y_1W^2=0
\end{equation}
in the projective bundle $X_2$.
$Z$ is rational since it is birational to a hypersurface in affine space $\A^5$ with coordinates $x_1,y_1,U,V,W$ by setting $y_0=x_0=1$. Since the equation is linear in $x_1$ (and $y_1$), the variety is rational.

The final component to study is the component $X_1$ defined by 
\[F(y_0,y_1,z_0,z_1,U,V,W,T)=0.\]
%
%
The variety $X_1$ is birational to the Hassett-Pirutka-Tschinkel quartic. We know that for the Hasset-Pirutka-Tschinkel quadric, the following holds: (c.f. \cite[Proposition 10]{HPT19}).
\begin{proposition}
	\label{prop:NontrivialClassHPT}
	Let $k=\C$ (\cite{HPT19}) or let $k$ be an algebraically closed field of characteristic different from 2 (\cite{SchSurvey}), let $\P_k^2 \times \P_k^3$ have coordinates $x,y,z$ and $s,t,u,v$ respectively. Let $K = k(x,y) = k(\P_k^2)$, and $Y \to \P_k^2$ be the following quadric surface:
	Then the hypersurface defined by:
	\[yzs^2 + xzt^2 + xyu^2 + F(x,y,z)v^2\]
	where
	\[F(x,y,z)=x^2+y^2+z^2-2(xy+xz+yz)\]
	is a quadric bundle over $\P_k^2$ with a non-trivial unramified cohomology class \[0 \neq \alpha \in H_{nr}^2(k(\P_k^2)/k,\mu_2^{\otimes 2}).\]
\end{proposition}
In \cite{HPT19}, the authors work over $\C$, but in \cite[Proposition 9.6]{SchSurvey} it is observed that the same proof works as long as $k$ is an algebraically closed field of characteristic different from 2. An immediate consequence is:
\begin{corollary}
	\label{cor:AlphaNonTriv}
	Let $k=\overline{\F_p}$ and $X_1$ be as above. Then there is a non-trivial class $0 \neq \alpha \in H_{nr}^2(k(X_1)/k,\mu_2^{\otimes 2})$
\end{corollary}
\begin{proof}
	If $Y$ is the quadric bundle defined in \cref{prop:NontrivialClassHPT}, then it is birational to $X_1$. To see this, note that after seting $z=1$ in the defining equation of $Y$, and $y_0=z_0=1$ in the defining equation for $X_1$ the equations are equal, so the varieties are birational. Therefore, $k(Y) \simeq k(X_1)$, so the corresponding unramified cohomology groups are also isomorphic.
\end{proof}

Similar to the Hassett-Pirutka-Tschinkel quartic, $X_1$ is also singular along four curves, two ``vertical'' curves (curves projecting to a point in $\P_{\F_p}^1 \times \P_{\F_p}^1$) defined by
\begin{gather*}
	y_1=z_1-y_0=U=V^2 + W^2 + T^2=0\\
	z_1=y_1-z_0=V=U^2 + W^2 + T^2=0
\end{gather*}
and two ``horizontal curves'', projecting to coordinate axes in $\P_{\F_p}^1 \times \P_{\F_p}^1$, defined by
\begin{gather*}
	z_1=V=T=y_1W^2+y_0U^2=0\\
	y_1=U=T=z_1W^2+y_0V^2=0
\end{gather*}
Importantly, we see that $Z$ meets the smooth locus of $X_1$. One can also compute that the rational variety $Z$ is singular along the same ``horizontal curves'' as $X_1$.

The following result by Schreieder will ensure that the singularities of $X_1$ don't interfere with the Merkurjev pairing.
\begin{theorem}{{\cite[Theorem 10.1]{SchSurvey}}}
  \label{thm:QuadricUnramified}
  Let $f \from Y \to S$ be a surjective morphism of proper varieties over an algebraically closed field $k$ with $char(k) \neq 2$ whose generic fibre is birational to a smooth quadric over $k(S)$. Let $n = \dim(S)$ and assume that there is a class $\alpha \in H^n(k(S),\mu_2^{\otimes n})$ with $f^* \alpha \in H_{nr}^n(k(Y)/k,\mu_2^{\otimes n})$.
  Then for any dominant generically finite morphism $\tau \from Y' \to Y$ of varieties and for any subvariety $E \subset Y'$ that meets the smooth locus of $Y'$ and which does not dominant $S$ via $f \circ \tau$, we have $\restr{(\tau^* f^* \alpha)}{E} = 0 \in H^n(k(E),\mu_2^{\otimes n})$.
\end{theorem}

We are now ready to prove that the special fibre does not have a decomposition of the diagonal. The proof is similar to the one found in \cite[Proposition 6.1]{SchTorsion}.

\begin{lemma}
 \label{lem:NoDoD}
	Let $X_p$ be the complete intersection $Q_p \cap C_p \subset \P_{\F_p}^6$ from \eqref{eq:Qquadric} and \eqref{eq:Qqubic}. Then $X_p$ does not admit a decomposition of the diagonal.
\end{lemma}
\begin{proof}
	First note that if a $X_p$ admits a decomposition of the diagonal, so does the base change of $X_p$ to $\overline{\F_p}$, the algebraic closure of $\F_p$. So in the remainder of the proof we will work over $k = \overline{\F_p}$. Let 
	\[\delta_{X_p} = z \times k({X_p}) \in CH_0({X_p} \times k({X_p}))\]
	be a decomposition of the diagonal of $X_p$ where $z$ is a zero-cycle on $X_p$.
	If we continue to let $X_1$ be the variety defined by the vanishing of \eqref{eq:X1} the map $X_1$ to $X_p$ is generically injective, so we can pull this relation back to $X_1$ and get the following equality:
	\begin{equation}
		\label{eq:DoDY}
		\delta_{X_1} = z_{X_1} \times k(X_1) + z' \times k(X_1) \in CH_0(X_1 \times k(X_1))
	\end{equation}
	for some $z'$ supported on $Z = X_1 \cap X_2$.
	
	Let $\tau \from X_1' \to X_1$ be an alteration of odd degree. Pulling back the equality \eqref{eq:DoDY} to $X_1'$ we get the equality:
	\begin{equation}
          \label{eq:DoDY'}
	\tau^*\delta_{X_1} = \tau^*z_{X_1} \times k(X_1) + \tau^*z' \times k(X_1) + z'' \times k(X_1) \in CH_0(X_1 \times k(X_1))
	\end{equation}
where now $z''$ is a zero-cycle supported on $\tau^{-1} X_1^{sing}$.

We now wish to compute the pairing of $\tau^* \alpha$ with both sides of \eqref{eq:DoDY'}, where $\alpha$ is the non-trivial class from \cref{cor:AlphaNonTriv}.
Computing the pairing of $\tau^*\alpha$ with the left hand side can be done as follows:
	\[\langle \tau^*\delta_{X_1}, \tau^* \alpha \rangle = \langle \tau_*\tau^*\delta_{X_1}, \alpha \rangle = (\deg \tau) \langle \delta_{X_1}, \alpha \rangle = (\deg \tau) \alpha \neq 0 \]

      On the other hand, the pairing of $\tau^* \alpha$ with each term of the right hand side of \eqref{eq:DoDY'} is zero. We show this term by term. Firstly, since we are working over the algebraically closed field $\overline{\F_p}$, $\tau^* \alpha$ vanishes when restricted to closed points. Hence $\langle \tau^*z \times k(X_1), \tau^*\alpha \rangle = 0$. For the second term, note that since $\tau^* z'$ is supported on $Z'$, it suffices to prove that the restriction of $\tau^* \alpha$ to $Z'$ is zero. Now, observe that even though $X_1$ is not smooth, the restriction $\restr{\alpha}{Z}$ is still defined since $Z$ meets the smooth locus of $X_1$. To compute the restriction of $\tau^* \alpha$ to $Z'$ we therefore consider the diagram:
\begin{equation*}
	\begin{tikzcd}
		Z' \arrow["\tau_Z"]{d} \arrow["i'"]{r} & X_1' \arrow["\tau"]{d} \\
		Z \arrow["i"]{r} & X_1
	\end{tikzcd}
\end{equation*}
from which we see that $\restr{\tau^*(\alpha)}{Z'} = \tau_Z^*(\restr{\alpha}{Z})$, and this is an unramified class, since $X_1'$ is smooth. The pushforward $(\tau_Z)_* \from H^2(k(Z),\mu_2) \to H^2(k(Z),\mu_2)$ takes unramified classes to unramified classes, so $(\deg \tau) (\restr{\alpha}{Z})$ is unramified, and therefore zero, since $Z$ is a rational variety. Since $(\deg \tau)$ is odd and the order of $\alpha$ is 2, we conclude that $\restr{\alpha}{Z}=0$, and therefore also $\restr{\tau^*\alpha}{Z'}=0$. For the last term, $\langle z'',\tau^*\alpha \rangle = 0$ since by \cref{thm:QuadricUnramified}, $\tau^*\alpha$ restricted to any subvariety of $\tau^{-1} X_1^{sing}$ is zero.

So in \eqref{eq:DoDY'} the pairing of $\tau^*\alpha$ with the left hand side is $\tau^*\alpha \neq 0$, but the pairing with the right hand side is 0, a contradiction. Therefore, $X_p$ cannot admit a decomposition of the diagonal.
\end{proof}	

Using this we can apply \cref{prop:ZSpecialization} to get the main result of the paper:

\begin{theorem}
  Let $R=\Z$ or $R=\F_p[t]$ with field of fractions $K$. In the first case let $p,q \geq 5$ be distinct primes and set $u=p,v=q$, and in the second case let $u=t,v=(t-1)$. Let $X$ be the smooth complete intersection in $\P^6_{K}$ defined by the intersection $Q \cap C$ for
\begin{equation*}
	\label{eq:Quadric3}
	Q = V\left( u(\sum_{i=0}^6 x_i^2) + v(x_3x_6-x_4x_5)\right) 
\end{equation*}
\begin{equation*}
	\label{eq:Cubic3}
	\begin{split}
		\mathscr{C} = &V\left( u(\sum_{i=0}^6 x_i^3) + v(x_0^2x_5 + x_1^2x_4 + x_2^2x_6 + x_3(x_5^2+x_4^2+x_3^2 -2x_3(x_6 + x_5 + x_4))) \right)
	\end{split}
\end{equation*}
Then $X$ is not stably rational.
\end{theorem}
\begin{proof}
	By \cref{lem:NoDoD}, $X$ specializes to a variety $Y_{\F_p}$ that does not admit a decomposition of the diagonal. But by \cref{prop:ZSpecialization} this can only happen if $X$ is not geometrically stably rational.
\end{proof}

\printbibliography
\typeout{get arXiv to do 4 passes: Label(s) may have changed. Rerun}
\end{document}